\documentclass[11pt]{article}

\usepackage[utf8]{inputenc} 
\usepackage{amsmath}
\usepackage{graphicx}
\usepackage{subcaption}
\usepackage[top=30mm, bottom=30mm, left=27mm, right=27mm]{geometry}
\usepackage{amsfonts}
\usepackage{mathtools}
\usepackage{subcaption}
\usepackage{amssymb}
\usepackage{amsthm}
\usepackage{chngpage}
\usepackage{hyperref}
\hypersetup{
	colorlinks=true,
	linkcolor=black,
	citecolor=black,
	filecolor=black,      
	urlcolor=black,
}
\usepackage{enumerate}
\usepackage{makeidx}
\usepackage{bm}
\usepackage{thmtools}
\usepackage{thm-restate}

\makeindex

\makeatletter
\renewenvironment{proof}[1][\proofname] {\par\pushQED{\qed}\normalfont\topsep6\p@\@plus6\p@\relax\trivlist\item[\hskip\labelsep\bfseries#1\@addpunct{.}]\ignorespaces}{\popQED\endtrivlist\@endpefalse}
\makeatother

\newcommand{\NN}{\mathbb{N}}
\newcommand{\PP}{\mathcal{P}}

\newcommand{\FF}{\mathcal{F}}
\newcommand{\GG}{\mathcal{G}}

\newtheorem{proposition}{Proposition}[section]
\newtheorem{lemma}[proposition]{Lemma}
\newtheorem{theorem}[proposition]{Theorem}

\newtheorem{claim}[proposition]{Claim}
\usepackage{amsthm}

\theoremstyle{definition}

\newtheorem*{remark*}{Remark}
\newtheorem*{theorem*}{Theorem}

\mathcode`l="8000
\begingroup
\makeatletter
\lccode`\~=`\l
\DeclareMathSymbol{\lsb@l}{\mathalpha}{letters}{`l}
\lowercase{\gdef~{\ifnum\the\mathgroup=\m@ne \ell \else \lsb@l \fi}}%
\endgroup

\title{A note on saturation for $k$-wise intersecting families}
\author{Barnabás Janzer\thanks{Department of Pure Mathematics and Mathematical Statistics, University of Cambridge, Wilberforce Road, Cambridge CB3 0WB, United Kingdom. Email: bkj21@cam.ac.uk. This work was supported by EPSRC DTG.}
}
\date{\vspace{-21pt}}

\begin{document}
	\maketitle
\begin{abstract}
	A family $\mathcal{F}$ of subsets of $\{1,\dots,n\}$ is called $k$-wise intersecting if any $k$ members of $\mathcal{F}$ have non-empty intersection, and it is called maximal $k$-wise intersecting if no family strictly containing $\mathcal{F}$ satisfies this condition. We show that for each $k\geq 2$ there is a maximal $k$-wise intersecting family of size $O(2^{n/(k-1)})$. Up to a constant factor, this matches the best known lower bound, and answers an old question of Erd\H{o}s and Kleitman, recently studied by Hendrey, Lund, Tompkins, and Tran.
\end{abstract}
\section{Introduction}

Given positive integers $k\geq 2$ and $n$, we say that a family $\FF$ of subsets of $[n]=\{1,2,\dots,n\}$ is \textit{$k$-wise intersecting} if whenever $X_1,\dots,X_k\in\FF$ then $X_1\cap\dots\cap X_k\not =\emptyset$.
It is well known (and easy to see) that any ($2$-wise) intersecting family over $[n]$ has size at most $2^{n-1}$, so the largest possible size of a $k$-wise intersecting family is clearly $2^{n-1}$ for all $k$. This is achieved, for example, by taking $\FF=\{A\in\PP([n]):1\in A\}$, where $\PP(X)$ denotes the set of subsets of $X$.

However, the corresponding saturation problem of finding the smallest possible size of a maximal $k$-wise intersecting family is more interesting for $k\geq 3$. (A family $\FF$ of subsets of $[n]$ is \textit{maximal $k$-wise intersecting} if it is $k$-wise intersecting but no family $\FF'$ over $[n]$ strictly containing $\FF$ is $k$-wise intersecting.) This problem was originally raised by Erdős and Kleitman \cite{erdos1974extremal} in 1974. Very recently, Hendrey, Lund, Tompkins, and Tran \cite{hendrey2021maximal} studied this problem for $k=3$. They determined the smallest possible size of a maximal $3$-wise intersecting family exactly when $n$ is sufficiently large and even. For general $k$, they showed that the smallest possible size $f_k(n)$ of a maximal $k$-wise intersecting family satisfies $c_k\cdot 2^{n/(k-1)}\leq f_k(n)\leq d_k\cdot 2^{n/\lceil k/2\rceil}$ (for some constants $c_k,d_k>0$). They asked about closing the exponential gap between the lower and upper bounds.

In this note we prove the following result, which shows that the lower bound gives the right order of magnitude.

\begin{theorem}\label{thm_kwise}
	For each $k\geq 3$ there exists some constant $C_k$ such that for all $n$ there is a maximal $k$-wise intersecting family over $[n]$ of size at most $C_k\cdot2^{n/(k-1)}$.
\end{theorem}

In the case when $n\geq 2(k-1)$ is a multiple of $k-1$, the exact value of our upper bound will be $2^{n/(k-1)+k-3}(k-1)-(2^{k-1}-1)(k-2)$. In the special case $k=3$ this upper bound is tight for $n$ sufficiently large (as shown by Hendrey, Lund, Tompkins, and Tran \cite{hendrey2021maximal}), but for $k\geq 4$ the construction has more complicated structure than for $k=3$. In fact, in \cite{hendrey2021maximal} it was shown that for $k=3$ (and $n$ large and even), the unique maximal $3$-wise intersecting families of smallest possible size are given by $\{A^c: A\in (\PP(X)\setminus\{X\})\cup(\PP(Y)\setminus\{Y\})\}$ for some partition $X\cup Y$ of $[n]$ into two equal parts. This was proved by first obtaining a stability result stating that for any `small' maximal $3$-wise intersecting family $\FF$, the family $\bar{\FF}=\{A^c:A\in \FF\}$ must be `close' to the union of two cubes $\PP(X)\cup\PP(Y)$ (with $X,Y$ as above).

However, the following result shows that, for $k\geq4$, directly generalising this approach cannot work, and it is necessary to have more complicated structure.

\begin{lemma}\label{lemma_closetocubes}
	Let $k\geq 4$, let $X_1\cup\dots \cup X_{k-1}$ be a partition of $[n]$ with each $X_i$ having size $n/(k-1)+O(1)$, and let $Q=\PP(X_1)\cup\dots\cup\PP(X_{k-1})$. If $|\FF\setminus Q|=o(2^{n/(k-1)})$ and $n$ is large enough, then $\bar{\FF}=\{A^c:A\in\FF\}$ cannot be maximal $k$-wise intersecting.
\end{lemma}

We mention that many other saturation problems have already been studied in the context of set systems and intersection properties. For example, several authors gave bounds for the smallest possible size $m(r)$ of a set system which is maximal among ($2$-wise) intersecting families $\FF\subseteq \binom{\NN}{r}$ consisting of sets of size $r$ -- see, for example, \cite{dow1985lower,blokhuis1987more,boros1989maximal}. A linear lower bound follows from a result of Erdős and Lovász \cite{erdHos1973problems}, and Dow, Drake, Füredi, and Larson \cite{dow1985lower} showed that in fact $m(r)\geq 3r$ for $r\geq 4$. Blokhuis \cite{blokhuis1987more} proved a polynomial upper bound of $m(r)\leq r^5$, and for certain values of $r$ quadratic upper bounds are also known -- see, e.g., \cite{blokhuis1987more,boros1989maximal}. Finding the order of magnitude of $m(r)$ is still an open problem. See the introduction and the references in \cite{hendrey2021maximal} for other related saturation problems.

\section{The construction}

We now prove Theorem \ref{thm_kwise} by describing a family of size $O(2^{n/(k-1)})$ and showing that it is maximal $k$-wise intersecting over $[n]$. For simplicity, we will work with complements, using the observation that $\bar{\GG}=\{X^c:X\in\GG\}$ is $k$-wise intersecting if and only if no $k$ elements of $\GG$ have union $[n]$.

Fix some $k\geq 3$ and $n\geq 2(k-1)$ integers. Partition $[n]$ into $k-1$ sets $A_1,\dots,A_{k-1}$ which are as close as possible in size, and pick `special' elements $a_i\in A_i$ for each $i$. Consider the following families. (All indices will be understood mod $k-1$, so, for example, $a_0=a_{k-1}$.) \begin{align*}
\FF_1(i)&=\PP(A_i)\setminus\{A_i\}\\
\FF_2(i)&=%\{X\cup\{a_{\ell}:\ell\in S\}:X\subseteq A_i, S\subseteq \{i+1,i+2,\dots,i+k-3\}, X\not =A_i\setminus\{a_i\}, X\not=A_i\}\\
\{X\cup Y: X\subseteq A_i, Y\subseteq \{a_1,a_2,\dots,a_{k-1}\}\setminus\{a_{i-1},a_i\}, X\not =A_i\setminus\{a_i\}, X\not=A_i\}\\
\FF&=\bigcup_{i=1}^{k-1}(\FF_1(i)\cup \FF_2(i)).
\end{align*}

We will show that $\{A^c:A\in \FF\}$ is maximal $k$-wise intersecting. Note that if $k=3$ then $\FF_2(i)\subseteq \FF_1(i)$ for each $i=1,2$, so $\FF$ is simply $\FF_1(1)\cup\FF_1(2)=(\PP(A_1)\cup\PP(A_2))\setminus\{A_1,A_2\}$. This was shown to be (up to isomorphism) the unique minimal-sized construction when $k=3$ and $n$ is even and sufficiently large by Hendrey, Lund, Tompkins, and Tran \cite{hendrey2021maximal}. Furthermore, note that  $$|\FF|=2^{{n}/{(k-1)}+k-3}(k-1)-(2^{k-1}-1)(k-2)$$ if $k-1$ divides $n$. Indeed, the number of subsets of $\{a_1,\dots,a_{k-1}\}$ appearing in $\FF$ is $2^{k-1}-1$, the number of sets which are not subsets of $\{a_1,\dots,a_{k-1}\}$ but appear in some $\FF_2(i)$ is $(k-1)\cdot (2^{n/(k-1)}-4)\cdot 2^{k-3}$, and finally, the only elements of $\FF$ we have not yet counted are the $(k-1)$ sets $A_i\setminus\{a_i\}$ for $i=1,\dots,k-1$. Summing these contributions gives the formula above.

\begin{claim}\label{claim_intersecting}
	The family $\FF$ contains no $k$ elements having union $[n]$.
\end{claim}
\begin{proof}
Suppose we have sets $Y_1,\dots,Y_k\in \FF$ satisfying $Y_1\cup\dots\cup Y_k=[n]$. Then clearly at least one of them must come from some $\FF_2(i)$, we may assume $Y_1\in\FF_2(1)$. If there is some $j\not=1$ and $i\not =1$ such that $Y_j\in \FF_2(i)$, then for each $t$ we can pick $b_t\in A_t\setminus\{a_t\}$ such that $b_t\not \in Y_1\cup Y_j$. Then no element of $\FF$ contains more than one $b_t$, but $\{b_1,\dots,b_{k-1}\}\subseteq\bigcup_{\ell\not=1,j}Y_\ell$, giving a contradiction. On the other hand, if there is no such pair $(i,j)$, then  $\{Y_\ell:\ell\not=1\}$ must contain at least one non-empty set from $\FF_1(t)$ (to cover $A_t\setminus\{a_t\}$) for $t=2,\dots,k-2$, at least two different sets from $\FF_1(k-1)$ (to cover $A_{t-1})$, and at least one set having an element in $A_1$, again giving a contradiction since these $k$ sets must all be different.
\end{proof}
\begin{claim}\label{claim_maximal}
	For any $X\in \PP([n])\setminus\FF$ there exist $X_1,\dots,X_{k-1}\in\FF$ such that $X_1\cup X_2\cup\dots\cup X_{k-1}\cup X=[n]$.
\end{claim}

\begin{proof}
		We first consider the following five cases, then check that each choice of $X$ belongs to at least one of these cases.
\begin{itemize}

	\item{Case 1:} $X\cap A_i\not =\emptyset$ for all $i$. Then let $X_i=A_i\setminus X$, so $X_i\in \FF_1(i)$ and the $X_i$ satisfy the conditions.
	\item{Case 2:} There is some $i$ such that $X\cap A_i$ contains an element $b_i$ with $b_i\not =a_i$ and $X\cap A_{i-1}\not =\emptyset$. We may assume that $i=1$. Then let $X_1=(A_1\setminus\{b_1\})\cup\{a_2,a_3,\dots,a_{k-2}\}$ (so $X_1\in\FF_2(1)$), let $X_j=A_j\setminus \{a_j\}$ for $j=2,\dots,k-2$ (so $X_j\in \FF_1(j)$) and let $X_{k-1}=A_{k-1}\setminus X$ (so $X_{k-1}\in \FF_{1}(k-1)$). These clearly satisfy the conditions.
	\item{Case 3:} There exist $i,j, b_i,b_j$ such that $i\not=j$, $b_i\in X\cap (A_i\setminus\{a_i\})$, $b_j\in X\cap (A_j\setminus\{a_j\})$. Then let $X_i=(A_i\setminus X)\cup\{a_{\ell}:\ell\not =i,i-1\}$, $X_j=(A_j\setminus X)\cup \{a_{\ell}:\ell\not=j,j-1\}$, and $X_\ell=A_\ell\setminus \{a_\ell\}$ for $\ell\not =i,j$. Then $X_i\in\FF_2(i), X_j\in \FF_2(j)$, and $X_\ell\in \FF_1(\ell)$ for $\ell\not =i,j$, so it is easy to see that the conditions are satisfied.
	\item{Case 4:} $X\supseteq (A_i\setminus\{a_i\})\cup \{a_j\}$ for some $i,j$ with $j\not =i$. Then let $X_i=\{a_t:t\not =j\}$ (so $X_i\in \FF_2(j+1)$), and let $X_{\ell}=A_\ell\setminus\{a_\ell\}$ for $\ell\not =i$ (so $X_\ell \in \FF_1(\ell)$). It is easy to see that the conditions are satisfied.
	\item{Case 5:} $X=A_i$ for some $i$. Then let $X_i=\{a_t:t\not=i\}$ (so $X_i\in \FF_2(i+1)$) and $X_\ell=A_\ell\setminus\{a_\ell\}$ for $\ell\not =i$ (so $X_\ell \in \FF_1(\ell)$). Then the conditions are again satisfied.
\end{itemize}

\bigskip
We now check that any $X\in \PP([n])\setminus\FF$ belongs to at least one of these cases. If $X\subseteq \{a_1,\dots,a_{k-1}\}$ then in fact $X=\{a_1,\dots,a_{k-1}\}$ and we are in Case 1. Otherwise there is some $i$ and some $b_i\in A_i$ such that $b_i\in X$ and $b_i\not  =a_i$. If $X\cap X_{i-1}\not=\emptyset$ then we are in Case 2. Otherwise, if $X$ is not a subset of $A_i\cup\{a_{\ell}:\ell\not=i,i-1\}$ then we are in Case 3. Finally, if $X$ is a subset of $A_i\cup\{a_{\ell}:\ell\not=i,i-1\}$, then we are in Case 4 or Case 5 since $X\not\in\FF_1(i)\cup\FF_2(i)$.
\end{proof}

\begin{proof}[Proof of Theorem \ref{thm_kwise}]
	By Claims \ref{claim_intersecting} and \ref{claim_maximal}, the family $\bar{\FF}=\{X^c:X\in \FF\}$ is maximal $k$-wise intersecting for all $k\geq 3$ and $n\geq 2(k-1)$. However, $\FF_1(i)$ and $\FF_2(i)$ both have size $O(2^{n/(k-1)})$ for each $i$, so $|\FF|=O(2^{n/(k-1)})$. The result follows.
\end{proof}

\section{Non-existence of certain types of maximal families}
Finally, we prove Lemma \ref{lemma_closetocubes} stating that there can be no construction for $k\geq4$ which is close to the union of $k-1$ cubes.

\begin{proof}[Proof of Lemma \ref{lemma_closetocubes}]
	Suppose that $\FF$ is as in the statement of the lemma, and $\bar{\FF}$ is maximal $k$-wise intersecting. Then at least one $X_i$ does not appear in $\FF$, we may assume $X_{k-1}\not \in \FF$. Let $\GG$ be the family of sets $S$ over $[n]$ satisfying the following conditions.
	\begin{itemize}
		\item The set $S$ cannot be written as $S_1\cup \dots\cup S_{k-1}$ such that $S_i\in \FF$ for all $i$ and $S_1\in \FF\setminus Q$.
		\item For all $S'\in \FF\setminus Q$ and $i\leq k-2$, we have $S'\cap X_i\not =S\cap X_i$.
		%	\item $S^c\cap X$ and $S^c\cap Y$ are non-empty.
		\item For all $i$, $S\cap X_i$ is non-empty.
		\item We have $S^c\cap (X_1\cup\dots\cup X_{k-2})\not \in \FF$.
	\end{itemize}
	It is easy to see that $|\GG|=(1-o(1))2^n$. Pick any $S\in \GG$, and let $T=S^c\cap(X_1\cup\dots\cup X_{k-2})$. Then $T\not \in \FF$, so (by maximality) there are $T_1,\dots,T_{k-1}\in\FF$ such that $T\cup T_1\cup T_2\cup\dots\cup T_{k-1}=[n]$. Furthermore, by maximality, $\FF$ must be a down-set, so we may assume that in fact $T^c=T_1\cup \dots\cup T_{k-1}$. So $S\cup X_{k-1}= T_1\cup \dots\cup T_{k-1}$. Write $S_i=T_i\setminus(X_{k-1}\setminus S)$, then $S=S_1\cup \dots\cup S_{k-1}$ and $S_i\in \FF$ for all $i$ (as $\FF$ is a down-set). So we must have $S_1,\dots,S_{k-1}\in Q\cap\FF$.
	Since $S\cap X_i$ is non-empty for all $i$, we may assume that $S_i=S\cap X_i$ for all $i$. Then, for $i\leq k-2$, $T_i\cap X_i=S\cap X_i$, so $T_i\not \in \FF\setminus Q$. Hence $T_i\in Q$ and $T_i=S_i$. But then $T_{k-1}=X_{k-1}$, giving a contradiction.
\end{proof}

\bibliography{Bibliography}

\begin{thebibliography}{1}

\bibitem{blokhuis1987more}
A.~Blokhuis.
\newblock More on maximal intersecting families of finite sets.
\newblock {\em Journal of Combinatorial theory, series A}, 44(2):299--303,
  1987.

\bibitem{boros1989maximal}
E.~Boros, Z.~F{\"u}redi, and J.~Kahn.
\newblock Maximal intersecting families and affine regular polygons in {$PG (2,
  q)$}.
\newblock {\em Journal of Combinatorial Theory, Series A}, 52(1):1--9, 1989.

\bibitem{dow1985lower}
S.~J. Dow, D.~A. Drake, Z.~F{\"u}redi, and J.~A. Larson.
\newblock A lower bound for the cardinality of a maximal family of mutually
  intersecting sets of equal size.
\newblock In {\em Proceedings of the sixteenth Southeastern international
  conference on combinatorics, graph theory and computing}, volume~48, pages
  47--48, 1985.

\bibitem{erdos1974extremal}
P.~Erd\H{o}s and D.~J. Kleitman.
\newblock Extremal problems among subsets of a set.
\newblock {\em Discrete Mathematics}, 8(3):281--294, 1974.

\bibitem{erdHos1973problems}
P.~Erd{\H{o}}s and L.~Lov{\'a}sz.
\newblock Problems and results on 3-chromatic hypergraphs and some related
  questions.
\newblock In {\em Colloquia Mathematica Societatis Janos Bolyai 10. Infinite
  and Finite Sets, Keszthely (Hungary)}, 1973.

\bibitem{hendrey2021maximal}
K.~Hendrey, B.~Lund, C.~Tompkins, and T.~Tran.
\newblock Maximal $3 $-wise intersecting families.
\newblock {\em arXiv preprint arXiv:2110.12708}, 2021.

\end{thebibliography}
\bibliographystyle{abbrv}
\end{document}